\documentclass[a4paper,11pt]{article}
\usepackage{array}
\usepackage{amsmath,amscd,amssymb,amsthm} 
\usepackage{latexsym}
\usepackage{epsfig}
\usepackage{enumerate}

\theoremstyle{plain}
\newtheorem{lemma}{Lemma}[section]

\newtheorem{proposition}[lemma]{Proposition}
\newtheorem{theorem}[lemma]{Theorem}
\newtheorem{corollary}[lemma]{Corollary}
\newtheorem{definition}[lemma]{Definition}
\theoremstyle{definition}
\newtheorem{example}[lemma]{Example}
\theoremstyle{remark}
\newtheorem{remark}[lemma]{Remark}
\newcommand{\FF}{\mathbb F}
\newcommand{\NN}{\mathbb N}
\newcommand{\PP}{\mathbb P}
\newcommand{\ZZ}{\mathbb Z}
\newcommand{\cB}{\mathcal B}
\newcommand{\cE}{\mathcal E}
\newcommand{\cF}{\mathcal F}
\newcommand{\cL}{\mathcal L}
\newcommand{\cO}{\mathcal O}
\newcommand{\cV}{\mathcal V}
\newcommand{\imic}{\cong}
\newcommand{\To}{\longrightarrow}
\newcommand{\Pic}{\mathop{\mathrm {Pic}}\nolimits}
\newcommand{\Proj}{\mathop{\null\mathrm {Proj}}\nolimits}

\newcommand{\Sym}{\mathop{\mathrm {Sym}}\nolimits}
\newcommand{\Hom}{\mathop{\mathrm {Hom}}\nolimits}
\newcommand{\Reg}{\mathop{\mathrm {Reg}}\nolimits}
\renewcommand{\Im}{\mathop{\mathrm {Im}}\nolimits}
\newcommand{\bdp}{\mathbf p}
\newcommand{\bdc}{\mathbf c}
\newcommand{\bdm}{\mathbf m}
\newcommand{\bdh}{\mathbf h}
\newcommand{\bdf}{\mathbf f}
\newcommand\tw[2]{\!\left[\begin{smallmatrix}{#1}\\{#2}\end{smallmatrix}\right]}

\newcommand{\Hyp}{{\bdh}}
\newcommand{\fib}{{\bdf}}
\newcommand{\fdiv}{{F}}
\newcommand{\hafdiv}{{S}}
\newcommand{\hdiv}{{H}}
\newcommand{\bdcC}{{\bf C}}
\makeatletter
\renewcommand\@makefntext[1]{%
\setlength\parindent{1em}%
\noindent
\mbox{\@thefnmark}{#1}}
\makeatother
\begin{document}
\title{Castelnuovo-Mumford Regularity over Scrolls and Splitting Criteria}
\author{F.~Malaspina and G.K.~Sankaran
\vspace{6pt}\\
{\small  Politecnico di Torino}\\
{\small\it Corso Duca degli Abruzzi 24, 10129 Torino, Italy}\\
{\small\it e-mail: francesco.malaspina@polito.it}\\
\vspace{6pt}\\
{\small University of Bath}\\
 {\small\it Bath BA2 7AY , England}\\
{\small\it e-mail: G.K.Sankaran@bath.ac.uk}}
\maketitle
\def\thefootnote{\null}
\begin{center}{\large Dedicated to the memory of Gianfranco Casnati}
\end{center}

\begin{abstract}\noindent
We introduce and study a notion of Castelnuovo-Mumford regularity
suitable for scrolls obtained as projectivisations of sums of line
bundles on $\PP^m$. We show that this is a natural generalisation of
the well known regularity on projective and multiprojective spaces and
we prove Horrocks-type splitting criteria for vector bundles. 
\end{abstract}

\section*{Introduction}\label{sect:intro}

Castelnuovo-Mumford regularity was defined initially on projective
space, but many generalizations exist to other varieties, mostly
rational varieties. Typically, these all reduce to classical
Castelnuovo-Mumford regularity on $\PP^N$ but are otherwise defined
independently, depending on the class of varieties to be studied. In
particular, if some variety falls into more than one such class, the
competing definitions may not agree. Over the years, extensions of
this notion have been proposed to handle other ambient varieties
beyond projective space, including Grassmannians of lines \cite{am2},
quadrics \cite{bm3}, multiprojective spaces \cite{bm2,cm2, hw},
$n$-dimensional smooth projective varieties with an $n$-block
collection~\cite{cm2}, and weighted projective spaces \cite{ms}.

A very general definition, for all simplicial toric varieties (and
more) was given by Maclagan and Smith in \cite{McSm}. Because of its
wide scope, that definition is not always the most suitable in
particular cases, so it makes sense to consider the possibilities for
toric varieties of a special type, even though the definition of
\cite{McSm} applies to them.

We work with scrolls: that is, projectivisations of sums of line
bundles on $\PP^m$. This is a natural and widespread generalisation of
the classical notion of scroll (which is the case $m=1$). If
$X=\PP\cV$ is a scroll in this sense then $\rho(X)=2$ and the Picard
group has obvious generators $\fib$, the pullback of $\cO_{\PP^m}(1)$,
and $\Hyp$, the relatively ample line bundle $\cO_{\PP\cV}(1)$. In
fact it was shown in~\cite{Kl} that these are exactly the smooth toric
varieties of Picard rank~2.

These choices allow us to give a definition of Castelnuovo-Mumford
regularity that has many of the properties of regularity proved by
Mumford in \cite{Mu} for projective space.  In particular we show that
also for our notion of regularity, a regular coherent sheaf is
globally generated. We compare our version with the regularity of
Maclagan and Smith (with suitable choices) in this case.

In the second part of the paper we use the new notion of regularity to
prove splitting criteria for vector bundles analogous to those of
Horrocks (see \cite{Ho}) on projective spaces.  We compare our
splitting criteria with those obtained for the same varieties by Brown
and Sayrafi~\cite{BS} (recently extended to arbitrary smooth
projective toric varieties in \cite{Sa}).

Both the notion of regularity and the splitting criteria are simpler
for rational normal scrolls, i.e.\ when $m=1$.

Smooth toric varieties of low Picard rank were described more fully in
\cite{Bat}. Also in \cite{Bat} there is a characterisation of
polyscrolls (recursive definition: a point is a polyscroll and if $X$
is a polyscroll and $\cV=\bigoplus \cL_i$ is a direct sum of line
bundles on $X$ then $X'=\PP\cV$ is a polyscroll). It is possible that
our definitions and some of our results can be extended to
polyscrolls.

\section{Scrolls and Regularity}\label{sect:scrolls}

We fix a decomposible vector bundle $\cV=\bigoplus\nolimits_{i=0}^n
\cO_{\PP^m}(a_i)$ of rank $n+1$ on $\PP^m$: we assume throughout that
$a_0\le a_1\le\ldots\le a_n$. The associated projective space bundle
$X:=\PP\cV$ is by definition $\Proj(\Sym \cV)$, adopting the
notational conventions of \cite[Section~II.7]{Hartshorne}. The
associated line bundle $\cO_X(1)$ is relatively ample over $\PP^m$,
and is ample on $X$ if $a_0>0$. We put $c:=\sum_{i=0}^n a_i$ and we let
$\pi\colon \PP (\cV) \to \PP^m$ be the projection. We denote by $\Hyp$
and $\fib$ the classes in $\Pic X$ of $\cO_{\PP(\cV)}(1)$
and the pullback $\pi^*\cO_{\PP^m}(1)$, respectively.  

\begin{definition}\label{def:scroll}
If $X=\PP\cV$ with $\cV$ as above, we call the pair $(X,\cV)$ an
\emph{abstract scroll} or simply a \emph{scroll}. If $a_0\ge 0$
we say that $(X,\cV)$ is a \emph{semipositive scroll}, or a
\emph{positive scroll} if $a_0>0$.
\end{definition}

Note that $\Hyp$ depends on $\cV$ rather than only on $X$. If
$\cV'=\cV\otimes \cO_{\PP^m}(w)=\bigoplus \cO_{\PP^m}(a_i+w)$ then
$\PP\cV'\imic X$ but $c'=c+(n+1)w$ and $\Hyp'=\Hyp+w\fib$. In
particular, this, and the fact that $\Hyp$ is ample if and
only if $a_0>0$, allow us to recover $\cV$ given $X$ and $\Hyp$.

If $(X,\cV)$ is a positive scroll then $\Hyp$ is globally generated
and the image $\phi_{|\Hyp|}(X)$ is a geometric scroll, i.e.\ a
subvariety of some $\PP^N$ in which the fibres of $\pi$ appear as
linear projective subspaces.

For conciseness, if $\cF$ is a sheaf on $X$ we will often write
$\cF\tw{a}{b}:=\cF(a\Hyp+b\fib)$.

If $I\subseteq \{0,\dots ,n\}$ we write $|I|$ for the cardinality of
$I$ and we set $a_I=\sum_{i\in I}a_i$.

With this notation, we have $\omega_X \cong \cO_X\tw{-(n+1)}{c-1-m}$.

The following two easy lemmas are useful for computation.
\begin{lemma}\label{lem:vanishing}
Let $X$ be a scroll. 
\begin{enumerate}[(i)]
\item $H^i(X, \cO_X\tw{a}{b}) \cong H^i(\PP^m, \Sym^a\cV \otimes
  \cO_{\PP^m}(b))$ if $a\ge 0$;
\item $H^i(X, \cO_X\tw{a}{b})) =0$ if $-n\le a<0$;
\item $H^i(X, \cO_X\tw{a}{b}) \cong H^{n+m-i}(\PP^m, \Sym^{-a-n-1}\cV
  \otimes \cO_{\PP^m}(c-b-1-m))$ if $a<-n$.
\end{enumerate}
\end{lemma}
\begin{proof}
  See \cite[Exercise III.8.4]{Hartshorne}.
\end{proof}
\begin{lemma}\label{lem:Hor}
Again let $X$ be a scroll. Then, for $0<i<n+m=\dim X$,
  \begin{enumerate}[(i)]
    \item if $a\geq 0$, then $H^i(X, \cO_X\tw{a}{b})=0$ for any
      $b\geq -m$.
    \item if $a<-n$, then $H^i(X, \cO_X\tw{a}{b})=0$ for any $b< c$.
  \end{enumerate}
\end{lemma}
\begin{proof}
Both parts follow from \ref{lem:vanishing}(i).
\end{proof}

Recall the dual of the relative Euler exact sequence of a scroll $X$:
\begin{equation}\label{eq:Eulerseq}
0\To \Omega_{X|\PP^m}^1 (\Hyp) \To
\cB:=\bigoplus_{i=0}^n\cO_X(a_i\fib) \To \cO_X(\Hyp) \To 0,
\end{equation}
and so we have $\omega_{X|\PP^m} \cong \cO_X\tw{-(n+1)}{c}$. The long
exact sequence of exterior powers associated to \eqref{eq:Eulerseq} is
\begin{equation}\label{eq:Eulerexterior}
\begin{aligned}
0&\To \cO_X\tw{-n}{c} \To \wedge^n \cB((-n+1)\Hyp)
\stackrel{d_{n-1}}{\To}\wedge^{n-1}\cB((-n+2)\Hyp)
\stackrel{d_{n-2}}{\To}\\
&\cdots \stackrel{d_1}{\To} \cB \To \cO_X(\Hyp) \To 0.
\end{aligned}
\end{equation}
Now \eqref{eq:Eulerexterior} splits into
\begin{equation}\label{eq:Eulerexteriorshort}
0\To \Omega_{X|\PP^m}^i (i\Hyp)\To \wedge^i \cB \To
\Omega_{X|\PP^m}^{i-1}(i\Hyp) \To 0
\end{equation}
for each $i=1,\ldots, n$, and we have $\Im(d_i\otimes
\cO_S((i-1)\Hyp))\cong \Omega_{X|\PP^1}^i(i\Hyp)\subset \wedge^i \cB$.

We will often use the following exact sequences, obtained as pullbacks
of Koszul sequences from $\PP^m$:
\begin{equation}\label{eq:pbKoszul1}
0 \to \cO_X(-m\fib)\to \cO_X^{e_m}(-(m-1)\fib)
\to \cdots \to  \cO_X^{e_1} \to \cO_X(\fib)\to 0,
\end{equation}
with $e_j=\binom{m+1}{j}$, and
\begin{equation}\label{eq:pbKoszul2}
\begin{aligned}
0 &\to \cO_X\tw{-n}{c-m-1}\to \cO_X^{e_m}\tw{-n}{c-m} \to \cdots \to
\cO_X^{e_1}\tw{-n}{c-1}\to\\
&\to\bigoplus_{i=0}^n\cO_X\tw{-(n-1)}{c-a_i}\to\cdots\to
\bigoplus_{|I|=r}\cO_X\tw{-(n-r)}{c-a_I}\to\cdots\to\cB \to
\cO_X(\Hyp) \to 0.
\end{aligned}
\end{equation}

\subsection{Main definition}\label{subsect:maindef}

Our main definition is the following notion of regularity on a scroll
$(X,\cV)$. For $p,\,q\in\ZZ$ we set $\bdp=p\Hyp+q\fib$ and we write
$\cF(\bdp)=\cF\tw{p}{q}$.

\begin{definition}\label{def:pqreg}
A coherent sheaf $\cF$ on $X$ is said to be \emph{$(p,q)$-regular} if
\begin{enumerate}[(a)]
\item $h^{n+j}(\cF(\bdp)\tw{-n}{c-j-1})= 0$ for $0\leq j\leq m$ but
  $(n,j)\neq(0,0)$, and
\item $h^{i+j}(\cF(\bdp)\tw{-i}{i-j})=0$ for $0\leq j\leq m$ and
  $0\leq i< n$ but $(i,j)\not=(0,0)$.
\end{enumerate}

We will say \emph{regular} to mean $(0,0)$-regular.  We define the
\emph{regularity} of $\cF$, denoted $\Reg(\cF)$, to be the least
integer $p$ such that $\cF$ is $(p,0)$-regular. We set
$\Reg(\cF)=-\infty$ if there is no such integer.
\end{definition}

\begin{example}\label{ex:lowcases} Some special cases of this are familiar.
\begin{enumerate}[(i)]
\item If $m=0$, then $X=\PP^n$ and $\fib=0$. The conditions
  \ref{def:pqreg}(a) and \ref{def:pqreg}(b) reduce respectively to
  $h^n(\cF(\bdp)(-n\Hyp))=0$ and $h^i(\cF(\bdp)(-i\Hyp))=0$ for $1\leq
  i< n$, giving the usual notion of Castelnuovo-Mumford regularity on
  $\PP^n$.
\item If $n=0$ then $X=\PP(\cO_{\PP^m}(c))=\PP^m$ and
  $\Hyp=c\fib$. Indeed, $\omega_X\cong\cO_X\tw{-1}{c-1-m} =
  \cO_X(-c\fib+(c-1-m)\fib) = \cO_X((-1-m)\fib)$. Condition
  \ref{def:pqreg}(b) is vacuous, and \ref{def:pqreg}(a) reduces to
  $h^j(\cF(\bdp)((c-j-1)\fib))=0$ for $1\leq j\leq m$. So for $c=1$
  this again gives the usual notion of Castelnuovo-Mumford regularity
  on $\PP^m$, and for $c>1$ it gives a notion of Castelnuovo-Mumford
  regularity over the Veronese varieties $(\PP^m,\cO_{\PP^m}(c))$.
\item If $m=1$ and $n=2$ then $X$ is a rational normal scroll surface
  and the regularity of Definition~\ref{def:pqreg} agrees with the
  definition given in \cite{DM}. For $n>2$ see Remark~\ref{rk:rnsreg}
  below.
\item If $c=n+1$ then $X=\PP^n\times\PP^m$ and $\Hyp=[\cO(1,1)]$. The
  regularity of Definition~\ref{def:pqreg} agrees with the definition
  given in \cite{bm2}, since if $0\leq j\leq m$ and $0\leq i< n$ but
  $(i,j)\not=(0,0)$ then
\[
h^{n+m-j}(\cF(\bdp)\tw{-n}{c-1-m-j})=h^{n+m-j}(\PP^n\times\PP^m,
\cF(\bdp)(-n,-m-j))
\]
and
\[
h^{i+j}(\cF(\bdp)\tw{-i}{i-j})= h^{i+j}(\PP^n\times\PP^m,
\cF(\bdp)(-i,-j)).
\]
\end{enumerate}
\end{example}

\subsection{Positivity}\label{subsect:positivity}

Although the definitions and examples in
Subsection~\ref{subsect:maindef} make sense for arbitrary choices of
$a_i$, for applications it is usually necessary to have some
information about the positivity of $\Hyp$. Therefore we assume
henceforth that $X$ is a positive scroll, i.e.\ $a_0>0$, so that
$\Hyp$ is ample on $X$.

\begin{lemma}\label{lem:positivetwistvanishing}
If $\cF$ is a regular coherent sheaf on a positive scroll $X$, then
\[
  h^{n+m}(\cF\tw{a-n}{c-1-m+b})=0 \text{ for any }a,\,b\geq 0.
\]
\end{lemma}
\begin{proof}
From \eqref{eq:pbKoszul1} we get $h^{n+m}(\cF\tw{-n}{c-1-m+t})=0$ for
any $t\geq 0$. From \eqref{eq:pbKoszul2} tensored by
$\cF\tw{-(n-1)}{c-2}$ we get
$h^{n+1}\bigl(\cF\tw{-(n-1)}{c-2+t}\bigr)=0$ and again by
\eqref{eq:pbKoszul1} we obtain
$h^{n+1}\bigl(\cF\tw{-(n-1)}{c-2+t}\bigr)=0$ for $t\geq 0$. In the
same way $h^{n+1}(\cF\tw{a-n}{c-2+b})=0$ for any $a\geq 0$ and for any
$b\geq 0$.
\end{proof}

Notice that, if $\fdiv$ is a smooth divisor in $|\fib|$ (which exists
since $\cO_X(\fib)$ is globally generated) then
$\fdiv=\PP(\cO_{\PP^{m-1}}(a_0)\oplus\dots
\oplus\cO_{\PP^{m-1}}(a_n))$, so our notion of regularity is available
on $F$ too.

\begin{lemma}\label{lem:frestrictionregular}
If $\cF$ is a regular coherent sheaf on $X$, and $\fdiv$ a smooth
divisor in $|\fib|$, then $\cF_{|\fdiv}$ is a regular coherent sheaf on
$\fdiv$.
\end{lemma}

\begin{proof}
We may assume $n,\,m>0$. For $0\leq j\leq m-1$ we consider the exact
cohomology sequence
\[
H^{n+j}(\cF\tw{-n}{c-j-1}) \to H^{n+j}(\cF_{|\fdiv}\tw{-n}{c-j-1})
\to H^{n+j+1}(\cF\tw{-n}{c-j-2}).
\]
The first and the third terms vanish by condition
\ref{def:pqreg}(a), so the middle term vanishes.  Similarly,
for $1\leq i\leq n-1$ and $0\leq j\leq m-1$, we consider
\[
H^{i+j}(\cF\tw{-i}{i-j}) \to H^{i+j}(\cF_{|\fdiv}\tw{-i}{i-j}) \to
H^{i+j+1}(\cF\tw{-i}{i-j-1})
\]
and since the first and third terms vanish by condition
\ref{def:pqreg}(a), then the middle term vanishes. Hence
$\cF_{|\fdiv}$ is regular.
\end{proof}

\begin{lemma}\label{lem:0qregular}
If $\cF$ is a regular coherent sheaf on $X$,
then it is also $(0,q)$-regular for any $q\geq 0$.
\end{lemma}

\begin{proof}
We proceed by induction on $m$. The case $m=0$ is proved in
\cite{Mu}. Now we assume the result for $m-1$ and we prove it for $m$.
For $0\leq j\leq m-1$ and for any $t\ge 0$ we have the exact sequence
\[
H^{n+j}(\cF\tw{-n}{c-j-2+t})\to H^{n+j}(\cF\tw{-n}{c-j-1+t})\to
H^{n+j}(\cF_{|\fdiv}\tw{-n}{c-j-1+t}).
\]
If $q=1$ then first term vanishes by conditions \ref{def:pqreg}(a) and
the third term vanishes by Lemma~\ref{lem:frestrictionregular} and the
inductive hypothesis, so the middle term vanishes. By recursion on $t$
we get $H^{n+j}(\cF\tw{-n}{c-j-1+q})=0$.  Moreover
$H^{n+j}(\cF\tw{-n}{c-j-1+q})=0$ also for $j=m$ by
Lemma~\ref{lem:positivetwistvanishing}. Thus $\cF(q\fib)$ satisfies
the conditions \ref{def:pqreg}(a).

The proof that $\cF(q\fib)$ satisfies the condition \ref{def:pqreg}(b)
is similar. For $1\leq i\leq n$ and $0\leq j\leq m$, we have
\[
H^{i+j}(\cF\tw{-i}{i-j-1+t}) \to H^{i+j}(\cF\tw{-i}{i-j+t}) \to
H^{i+j}(\cF_{|\fdiv}\tw{-i}{i-j+t},
\]
for any $t$. If $t=1$ then first term vanishes by condition
\ref{def:pqreg}(b). We want to show that the third term also
vanishes. Since $\cF_{|\fdiv}$ is regular by
Lemma~\ref{lem:frestrictionregular}, and $\cF_{|\fdiv}(\fib)$ is
regular by the inductive hypothesis, we have
$H^{i+j}(\cF_{|\fdiv}\tw{-i}{i-j+1})=0$ for $1\leq i\leq n-1$ and
$0\leq j\leq m-1$, and also for $(i,j)=(n,m)$ since
$\dim(\fdiv)<n+m$. Moreover the regularity of $\cF_{|\fdiv}(\fib)$
implies $H^{i+j}(\cF_{|\fdiv}\tw{-i}{i-j+1})=0$ for $j=n$ and $0\leq
j\leq m-1$.

Thus the middle term vanishes in all relevant cases and again by
recursion on $t$ we get $H^{i+j}(\cF\tw{-i}{i-j+q})=0$ for $1\leq
i\leq n$ and $0\leq j\leq m$. Hence $\cF(q\fib)$ also satisfies
condition \ref{def:pqreg}(b), and hence $\cF$ is $(0,q)$-regular.
\end{proof}

Notice that, if $\hafdiv$ is a smooth divisor in $|\Hyp-a_0\fib|$
(again, $\cO_X(\Hyp-a_0\fib)$ is globally generated) then $\hafdiv=\PP
(\cO_{\PP^m}(a_1)\oplus\dots \oplus\cO_{\PP^m}(a_n))$, so again the
regularity of Definition~\ref{def:pqreg} is available. We continue to
use $\Hyp$ and $\fib$ for the generators of $\Pic(\hafdiv)$: in $X$
they are obtained as the intersection products $(\Hyp-a_0\fib)\Hyp$
and $(\Hyp-a_0\fib)\fib$.

When $n=1$, we have $\hafdiv=\PP (\cO_{\PP^m}(a_1))\imic\PP^m$, and
$(\Hyp-a_0\fib)\fib=\fib$, the ample generator $\cO_{\PP^m}(1)$ of
$\Pic \hafdiv$. Then $(\Hyp-a_0\fib)\Hyp=c\fib-a_0\fib=a_1\fib$, so on
$S$ we have $\Hyp=a_1\fib$.

When $m=0$, we have $\hafdiv=\PP^{n-1}$, simply a hyperplane section
in $\PP^n$.

\begin{lemma}\label{lem:hfrestrictionregular}
Suppose that $n>0$ and $m>0$. If $\cF$ is a regular coherent sheaf on
$X$ and $\hafdiv$ is a smooth divisor in $|\Hyp-a_0\fib|$, then
$\cF_{|\hafdiv}$ is regular on $\hafdiv$.
\end{lemma}

\begin{proof}
Consider the exact cohomology sequence, for $0\leq j\leq m$,
\[
H^{n-1+j}\bigl(\cF\tw{-(n-1)}{c-a_0-j-1}\bigr) \to
H^{n-1+j}\bigl(\cF_{|\hafdiv}\tw{-(n-1)}{c-a_0-j-1}\bigr) \to
H^{n+j}(\cF\tw{-n}{c-j-1}).
\]
The third term vanishes by condition \ref{def:pqreg}(a) and
Lemma~\ref{lem:positivetwistvanishing}, and the first term vanishes by
conditions \ref{def:pqreg}(b) and Lemma~\ref{lem:0qregular} (notice
that $c-a_0-j-1\geq n-1-j$ because $c-a_0=a_1+\dots +a_n\geq n$), so
the middle term vanishes, so conditions \ref{def:pqreg}(a) hold for
$\cF_{|\hafdiv}$.

For $1\leq i< n-1$ and $1\leq j\leq m$ but $(i,j)\neq(0,0)$, we
consider the exact sequence
\[
H^{i+j}(\cF\tw{-i}{i-j}) \to H^{i+j}(\cF_{|\hafdiv}\tw{-i}{i-j}) \to
H^{i+j+1}\bigl(\cF\tw{-(i+1)}{i-j+a_0}\bigr).
\]
The first term vanishes by \ref{def:pqreg}(b) and the third term
vanishes by Lemma~\ref{lem:0qregular} (notice that $i-j+a_0\geq
i+1-j$), so the middle term vanishes, so conditions \ref{def:pqreg}(b)
hold for $\cF_{|\hafdiv}$.
\end{proof}

\begin{proposition}\label{prop:spanning}
If $\cF$ is a regular coherent sheaf on $X$, then
\begin{enumerate}[(i)]
  \item $\cF\tw{p}{q}$ is regular for $p,\,q\geq 0$.
  \item $H^0(\cF(\fib))$ is spanned by (hence, equal to)
    $H^0(\cF)\otimes H^0(\cO(\fib))$, and $H^0(\cF(\Hyp))$ is spanned
    by $H^0(\cF(a_0\fib))\oplus\dots \oplus H^0(\cF(a_n\fib))$.
\end{enumerate}
\end{proposition}

\begin{proof}
To prove (i) it is enough to show that $\cF\tw{1}{0}=\cF(\Hyp)$ is
regular, since $\cF\tw{0}{1}=\cF(\fib)$ is regular by
Lemma~\ref{lem:0qregular}.

We have $H^{n+m}\bigl(\cF\tw{-(n-1)}{c-1-j}\bigr)=0$ by
Lemma~\ref{lem:positivetwistvanishing}. For $0\leq j\leq m-1$ we
consider the exact sequence
\[
H^{n+j}(\cF\tw{-n}{c+a_0-1-j}) \to
H^{n+j}\bigl(\cF\tw{-(n-1)}{c-1-j}\bigr) \to
H^{n-1+j+1}\bigl(\cF_{|\hafdiv}\tw{-(n-1)}{c-1-j}\bigr).
\]
The first term vanishes by hypothesis (and the fact that
$c+a_0-1-j>c-i-j)$ and the third term vanishes by
Lemma~\ref{lem:hfrestrictionregular} (and the fact that $c-1-j\geq
c-a_0-1-j-1$).  Thus the middle term vanishes, so \ref{def:pqreg}(a)
holds for $\cF(\Hyp)$.

We verify the condition \ref{def:pqreg}(b) for $\cF(\Hyp)$ by
induction on $n$. We need to show that if $0\leq i\leq n-1$ and $0\leq
j\leq m$ but $(i,j)\neq(0,0)$ then
$h^{i+j}\bigl(\cF\tw{-(i-1)}{i-j}\bigr)=0$.

If $n=1$, then $1\leq j\leq m$ then $\hafdiv=\PP^m$ and we have the
exact sequence
\[
H^j(\cF((a_0-j)\fib)) \to H^j(\cF(\Hyp-j\fib)) \to
H^j(\cF_{|\hafdiv}(\Hyp-j\fib))
\]
in which the first term vanishes by hypothesis (since $a_0-j>-j)$ and
the third term vanishes because
$H^j(\cF_{|\hafdiv}(\Hyp-j\fib))=H^j(\PP^m,\cF_{|\PP^m}(c-j))$, and
$\cF_{|\PP^m}(t)$ is regular on $\PP^m$ for any $t\geq 0$. Hence we
have the required vanishing of the middle term for $n=1$.

For the induction step, we may now assume the vanishing for
$\cF_{|\hafdiv}(\Hyp)$ on $\hafdiv$. Then, for $0\leq i\leq n-1$ and
$0\leq j\leq m$ but $(i,j)\not=(0,0)$, we have
\[
H^i(\cF\tw{-i}{i-j+a_0})\to H^i(\cF\tw{-(i-1)}{i-j}) \to
H^i(\cF_{|\hafdiv}\tw{-(i-1)}{i-j})
\]
in which the first term vanishes by \ref{def:pqreg}(b) and
Lemma~\ref{lem:0qregular}, since $i-j+a_0>i-j$, and the third term
vanishes by the inductive hypothesis and
Lemma~\ref{lem:hfrestrictionregular}. So the middle term vanishes, as
required.

For the proof of (ii), we start with \eqref{eq:pbKoszul1} and tensor
by $\cF$ to get
\[
0\to \cF^{e_{m}}(-(m-1)\fib) \to \cdots \to  \cF^{e_1} \to \cF(\fib)\to
0.
\]
Since $H^1(\cF(-\fib))=\dots =H^{m-1}(\cF(-(m-1)\fib))=0$ by
\ref{def:pqreg}(b) with $i=0$ and $j=1,\dots,m-1$, we obtain
\[
H^0(\cF)\otimes H^0(\cO_X(\fib))\to H^0(\cF(\fib)) \to 0,
\]
which gives the first part of~(ii) immediately. If instead we take
\eqref{eq:pbKoszul2} and tensor by $\cF$ we get:
\begin{equation*}
\begin{aligned}
0 &\to \cF\tw{-n}{c-m-1}\to\cF^{e_m}\tw{-n}{c-m}
\to\cdots\to\cF^{e_m}\tw{-n}{c-1} \to\bigoplus_{i=0}^n
\cF\tw{-(n-1)}{c-a_i}\to\\
&\cdots\to\bigoplus_{i=0}^n \cF(a_i\fib)
\to \cF(\Hyp) \to 0.
\end{aligned}
\end{equation*}
From \ref{def:pqreg}(a) we have
$H^{n+m}(\cF\tw{-n}{c-m-1})=\dots=H^n(\cF\tw{-n}{c-1})=0$.

From \ref{def:pqreg}(b), with $j=0$ and $i>0$, and using
Lemma~\ref{lem:0qregular}, we obtain $H^i(\cF\tw{-i}{a_I})=0$ for any
$I\subset\{0,\ldots,n\}$ with $|I|=i+1$, because then
$a_I>i=i-j$. From this we get immediately
\[
H^0(\cF(a_0\fib))\oplus\dots \oplus H^0(\cF(a_n\fib))\to
H^0(\cF(\Hyp)) \to 0
\]
as required.
\end{proof}

\begin{corollary}\label{cor:globalgen}
If $\cF$ is a regular coherent sheaf on $X$ then it is globally
generated.
\end{corollary}
\begin{proof}
Proposition~\ref{prop:spanning} gives surjections $H^0(\cF)^r\to
\bigoplus_{k=0}^n H^0(\cF(a_k\fib))$, for $r=h^0(\cO(\fib))$, and
$\bigoplus_{k=0}^n H^0(\cF(a_k\fib)) \to H^0(\cF(\Hyp))$. Thus we have
a surjection $H^0(\cF)^r\to H^0(\cF(\Hyp))$. Take $l\gg 0$ such that
$\cF(l\Hyp)$ is globally generated. For a suitable positive integer
$s$ the diagram
\[
\begin{CD}
  H^0(\cF)^s\otimes\cO_X @>>>   H^0(\cF(l\Hyp))\otimes\cO_X\\
  @VVV                         @VVV \\
  H^0(\cF)^r             @>>>   \cF(l\Hyp)
\end{CD}
\]
commutes and the left, top and right arrows are surjections. Therefore
the bottom arrow is also a surjection, so $\cF$ is generated by its
sections.
\end{proof}

\begin{corollary}\label{cor:reg0}
If $n>0$, then $\cO_X$, $\cO_X(\fib)$ and $\cO_X(\Hyp-\fib)$ are all
regular but not $(-1,0)$-regular, so
$\Reg(\cO_X)=\Reg(\cO_X(\fib))=\Reg(\cO_X(\Hyp-\fib))=0$.
\end{corollary}
\begin{proof}
For $\cO_X$ and $\cO_X(\fib)$ the conditions \ref{def:pqreg}(a), and
\ref{def:pqreg}(b) for $i\neq 0$, follow immediately from
Lemma~\ref{lem:vanishing}(ii). For $i=0$ we require for
\ref{def:pqreg}(b) that $H^j(\cO_X(-j\fib))=H^j(\cO((-j+1)\fib)=0$ for
$j\le m$, and both of these are cases of Lemma~\ref{lem:Hor}.

So $\cO_X$ and $\cO_X(\fib)$ are regular.  On the other hand, they are
not $(-1,0)$-regular because $\cO_X(-\Hyp)$ and $\cO_X(-\Hyp+\fib)$
are not regular, since
\[
h^{n+m}(\cO_X(-\Hyp)\tw{-n}{c-1-m})=h^n(\cO_X(-\Hyp+\fib)\tw{-n}{c-1})=1.
\]

In the case of $\cO_X(\Hyp-\fib)$, for \ref{def:pqreg}(a) we require
that $h^{n+j}(\cO_X\tw{-n+1}{c-j-2})=0$ if $0\le j\le m$. This holds
by \ref{lem:vanishing}(ii) if $n\neq 1$ and by Lemma~\ref{lem:Hor} if
$n=1$, since $c-j-2\geq -m$.

For \ref{def:pqreg}(b) we require that
$h^{i+j}(\cO_X\tw{-i+1}{i-j-1})= 0$, for $0\leq j\leq m$ and $0\leq i<
n$ but $(i,j)\neq(0,0)$. This holds by Lemma~\ref{lem:vanishing}(ii)
when $i>1$, and by Lemma~\ref{lem:Hor} when $i=1$ since $i-j-1\geq
-m$. For $i=0$ it holds because, by Lemma~\ref{lem:vanishing}(i), we
have $H^j(\cO_X(\Hyp-(j+1)\fib)=H^j(\PP^m,
\cV\otimes\cO_{\PP^m}(-j-1))=0$.

So $\cO_X(\Hyp-\fib)$ is regular. On the other hand, it is not
$(-1,0)$-regular because $\cO_X(-\fib)$ is not regular, as since
$h^m(\cO_X(-\fib)\otimes\cO_X(-m\fib))=1$.
\end{proof}

\subsection{Comparison with multigraded regularity}\label{subsect:otherregularities}

The regularity defined in Subsection~\ref{subsect:maindef} is related
to the multigraded regularity of Maclagan and Smith~\cite{McSm} as it
applies in this special case.

Multigraded regularity, as defined in \cite[Definition 6.2]{McSm},
makes sense on any smooth projective variety $Y$ (and far more
generally). It depends, as ours does, on a chosen $\bdp\in \Pic Y$,
but also on a choice of a finite subset
$\bdcC=\{\bdc_1,\ldots,\bdc_\ell\}\subset \Pic Y$. If we further
assume that the divisors $\bdc_i$ are nef,
then~\cite[Corollary~6.6]{McSm} gives the following very simplified
definition for multigraded regularity.

\begin{definition}\label{def:MSreg}
  Suppose that $\cF$ is a sheaf on a smooth projective toric variety $Y$ and
  $\bdp\in\Pic Y$, and $\bdcC=\{\bdc_1,\ldots,\bdc_\ell\}$ is a finite
  set of nef divisor classes on $Y$. Then we say that $\cF$ is
  \emph{multigraded $\bdp$-regular} (with respect to $\bdcC$) if
  $H^i(\cF(\bdp-\sum\lambda_r\bdc_r))=0$ for all $i>0$ whenever
  $\lambda_i\in\NN$ and $\sum\lambda_r=i$. 
\end{definition}

Comparison of~\cite[Corollary~6.6]{McSm}
with~\cite[Definition~6.2]{McSm}, and the discussion there, gives the
following.

\begin{lemma}\label{def:MSlemma}
If $\cF$ is multigraded $\bdp$-regular under the conditions
of~\ref{def:MSreg}, then we also have
$H^i(\cF(\bdp+\bdm-\sum\lambda_r\bdc_r))=0$ for any $\bdm=\sum
\mu_r\bdc_r$ with $\mu_r\in\NN$ (i.e.\ $\bdm\in\NN\bdcC$ in the
notation of~\cite{McSm}).
\end{lemma}

To compare this with our definition, we shall take $Y$ to be a
semipositive scroll $X$, and $\bdcC=\{\Hyp,\fib\}$. Note that $\fib$
is nef by construction but $\Hyp$ is nef because of the
semipositivity.  The condition in Definition~\ref{def:MSreg} may be
rewritten as
\begin{equation}\label{eq:MS2}
H^{i+j}(\cF(\bdp)\tw{-i}{-j})=0\text{\ \ for all\ }i,\,j\in\NN\text{\ except\ }i=j=0.
\end{equation}

\begin{proposition}\label{prop:compareMS}
Let $(X,\cV)$ be a semipositive scroll and suppose that $\cF$ is a
multigraded $\bdp$-regular coherent sheaf on $X$, where
$\bdp=p\Hyp+q\fib$. Then $\cF$ is $(p,q)$-regular in the sense of
Definition~\ref{def:pqreg}.
\end{proposition}

\begin{proof}
It is sufficient to use the condition~\eqref{eq:MS2}. Then applying
Lemma~\ref{def:MSlemma} with $\bdm=i\fib$ with $i<n$ and $j\le m$
gives Definition~\ref{def:pqreg}(b), and the same with $i=n$ and
$\bdm=(c-1)\fib$ gives Definition~\ref{def:pqreg}(a).
\end{proof}

On the other hand, if $X$ is the Hirzebruch surface
$\FF_2=\PP(\cO_{\PP^1}\oplus\cO_{\PP^1}(2))$ then $\cO_{\FF_2}$ is not
multigraded $(0,0)$-regular: see~\cite[Example~1.2]{McSm}
and~\cite[Example~6.5]{McSm}. However, $\cO_{\FF_2}$ is regular
according to Definition~\ref{def:pqreg}, by Corollary~\ref{cor:reg0}.

The main difference between the two definitions concerns the top
cohomology. Definition~\ref{def:MSreg} requires the vanishing
conditions
\[
H^{m+n}(\cF(\bdp)\tw{-r}{-m-n+r})=0\text{ for }0\le r\le m+n,
\]
which in this context reduce to the last one,
$H^{m+n}(\cF(\bdp)\tw{-m-n}{0})=0$. By contrast,
Definition~\ref{def:pqreg} requires only
$H^{m+n}(\cF(\bdp)\tw{-n}{c-m-1})=0$, which is a much weaker condition
(but makes sense only for scrolls).

In Section~\ref{sect:splitting} we will exploit this to prove
splitting criteria on scrolls, analogous to those of
Horrocks~\cite{Ho} on $\PP^n$.

\subsection{Rational normal scrolls}\label{subsect:rationalnormalscrolls}

We end this section with a discussion of the case $m=1$: these are the
rational normal scrolls $X=\PP (\cO_{\PP^{1}}(a_0)\oplus\dots
\oplus\cO_{\PP^{1}}(a_n))$.

\begin{remark}\label{rk:rnsreg}
If $X$ is a rational normal scroll and $\cF$ is a coherent sheaf on
$X$, then $\cF$ is $(p,q)$-regular if and only if
\begin{enumerate}[(a)]
\item $h^{n+1}(\cF\tw{-n}{c-2})=h^n(\cF(\bdp)\tw{-n}{c-1})=0$,
\item $h^{i+1}(\cF(\bdp)\tw{-i}{i-1})=0$ for $0\le i<n$,
\item $h^i(\cF(\bdp)\tw{-i}{i})= 0$ for $0<i<n$,
\end{enumerate}
where, as before, $\cF(\bdp)=\cF\tw{p}{q}$.
\end{remark}

In this case a smooth hyperplane section $\hdiv\in |\Hyp|$ of $X$ is
still a rational normal scroll of the same degree: i.e.\ 
$\hdiv=\PP(\cO_{\PP^1}(a'_0)\oplus\dots\oplus\cO_{\PP^1}(a'_{n-1}))$
with $a'_0+\dots +a'_{n-1}=c$. We continue to use $\Hyp$ and $\fib$
for the generators of $\Pic(\hdiv)$ (in $X$ they are obtained as the
intersection products $\Hyp^2$ and $\Hyp\fib$).

\begin{lemma}\label{lem:hrestrictionregular}
If $\cF$ is a regular coherent sheaf on the rational normal scroll $X$
and $\hdiv$ is a smooth divisor in $|\Hyp|$, then $\cF_{|\hdiv}$ is
$(0,1)$-regular on $\hdiv$.
\end{lemma}

\begin{proof}
In the exact sequence
\[
H^n\bigl(\cF\tw{-(n-1)}{c-1}\bigr) \to
H^n\bigl(\cF_{|\hdiv}\tw{-(n-1)}{c-1}\bigr) \to
H^{n+1}(\cF\tw{-n}{c-1})
\]
the third term vanishes by \ref{rk:rnsreg}(a) and
Lemma~\ref{lem:frestrictionregular} and the first term vanishes by
condition \ref{rk:rnsreg}(b) and Lemma~\ref{lem:0qregular}, since
$c-1\geq n-2$. So the middle term vanishes, so
$H^n(S,\cF_{|\hdiv}\tw{-(n-1)}{c-1}$. Similarly, in
\[
H^{n-1}(\cF\tw{-(n-1)}{c}) \to H^{n-1}(\cF_{|\hdiv}\tw{-(n-1)}{c}) \to
H^n(\cF\tw{-n}{c})
\]
the third and first terms vanish by \ref{rk:rnsreg}(a) and
Lemma~\ref{lem:0qregular}, respectively by \ref{rk:rnsreg}(c) and
Lemma~\ref{lem:0qregular}, since $c\geq n-1$. So the middle term
vanishes, so $H^{n-1}(\hdiv,\cF_{|\hdiv}\tw{-(n-1)}{c})=0$, and we
have verified \ref{rk:rnsreg}(a) for $\cF_{|\hdiv}(\fib)$.

For $0\leq i\leq n-2$, in
\[
H^{i+1}(\cF\tw{-i}{i}) \to H^{i+1}(\cF_{|\hdiv}\tw{-i}{i}) \to
H^{i+2}(\cF\tw{-(i+1)}{i})
\]
the first and the third terms vanish by \ref{rk:rnsreg}(b) and
Lemma~\ref{lem:0qregular}. So the middle term vanishes, giving
\ref{rk:rnsreg}(b) for $\cF_{|\hdiv}(\fib)$.

Finally, for $1\leq i\leq n-2$, in
\[
H^i(\cF\tw{-i}{i+1}) \to H^i(\cF_{|\hdiv}\tw{-i}{i+1}) \to
H^{i+1}\bigl(\cF\tw{-(i+1)}{i+1}\bigr)
\]
the first and third terms vanish by \ref{rk:rnsreg}(c) and
Lemma~\ref{lem:0qregular}. So the middle term vanishes, giving
\ref{rk:rnsreg}(c) for $\cF_{|\hdiv}(\fib)$.
\end{proof}

\section{Splitting criteria for vector bundles}\label{sect:splitting}

In this section we assume $m,\,n>0$. We apply the results of
Section~\ref{sect:scrolls} in order to prove splitting criteria for
vector bundles.

\begin{theorem}\label{thm:splittingO}
Suppose that $(X,\cV)$ is a positive scroll with $m,\,n>0$, and let
$\cE$ be a rank $r$ vector bundle on $X$. Then the following
conditions are equivalent:
\begin{enumerate}[(i)]
\item for any integer $t$ we have the vanishing
  \begin{enumerate}[(a)]
  \item $h^{n+j}(\cE\tw{t}{c-j-1})=0$ for $0\leq j< m$ and
  \item $h^{i+j}(\cE\tw{t}{i-j})=0$ for $0\leq j\leq m$ and $0\leq
    i<n$ but $(i,j)\neq(0,0)$;
  \end{enumerate}
\item there are $r$ integers $t_1,\dots,t_r$ such that $\cE\cong
  \bigoplus_{i=1}^r \cO_X(t_i\Hyp)$.
  \end{enumerate}
\end{theorem}

\begin{proof}
First suppose that $\cE$ satisfies (i), and let $t$ be an integer such
that $\cE(t\Hyp)$ is regular but $\cE((t-1)\Hyp)$ is not. Comparing
the definition of regularity (Definition~\ref{def:pqreg} with $p=q=0$)
with (i) we see that $\cE((t-1)\Hyp)$ is not regular if and only if
$H^{n+m}(\cE\tw{t-n-1}{c-m-1})\neq 0$. In that case, by Serre duality
$H^0(\cE^\vee(-t\Hyp))\neq 0$. But this, together with the fact that
$\cE(t\Hyp)$ is globally generated by Corollary~\ref{cor:globalgen},
implies that $\cO_X$ is a direct summand of $\cE(t\Hyp)$. By induction
on $r$, it follows that $\cE$ satisfies~(ii).
  
Conversely, if $\cE$ satisfies (ii) then it satisfies (i) because
$\cO_X$ satisfies all the conditions in (i), by Lemma~\ref{lem:Hor}.
\end{proof}

In the case $X\cong\PP^n\times\PP^m$, i.e.\ $c=n+1$,
Theorem~\ref{thm:splittingO} reduces to
\cite[Theorem~1.3]{bm2}. Indeed,
\[
h^{n+j}(\cE\tw{t}{c-j-1})=h^{n+j}(\cE(t,t+n-j))=h^{n+j}(\cE(t-n,t-j))
\]
for $0\leq j< m$, and
\[
h^{i+j}(\cE\tw{t}{i-j})=h^{i+j}(\cE(t,t+i-j))=h^{i+j}(\cE(t-i,t-j))
\]
for $0\leq j\leq m$ and $0\leq i< n$ but $(i,j)\neq(0,0)$.

\begin{theorem}\label{thm:splittingOfh}
Let $\cE$ be a vector bundle on $X$. Then $\cE$ is a direct sum of
line bundles $\cO_X$, $\cO_X(\fib)$ and $\cO_X(\Hyp-\fib)$ with some
twist $t\Hyp$ if and only if the following conditions hold for any
integer $t$.
\begin{enumerate}[(a)]
\item $h^{n+j}(\cE\tw{t}{c-j-1})= 0$ for $1\leq j< m$,
\item $h^{i+j}(\cE\tw{t}{i-j})= 0$ for $0\leq j\leq m$ and $0\leq i<
  n$ but $(i,j)\neq(0,0), (0,m)$,
\item $h^{j+1}(\cE^\vee\tw{t}{-j})=0$ for $0\leq j< m$,
\item $h^{|I|}(\cE\tw{t}{a_I-1})=h^{|I|}(\cE^\vee\tw{t}{a_I-1})= 0$ if
  $1\leq |I|\leq n$.
\end{enumerate}
\end{theorem}

\begin{proof}
As in Theorem \ref{thm:splittingO} it is easy to check that $\cO_X$,
$\cO_X(\fib)$ and $\cO_X(\Hyp-\fib)$ satisfy all the required
vanishing.

Assume that (a)--(d) hold, and consider the integer $t$ such that
$\cE(t\Hyp)$ is regular but $\cE((t-1)\Hyp)$ is not. Up to a twist we
may assume $t=0$.

Comparing (a) and (b) with the definition of regularity, we see that
$\cE(-\Hyp)$ is not regular if and only if one of the following
conditions is satisfied:
\begin{enumerate}[(i)]
\item $h^{n+m}\bigl(\cE\tw{-(n+1)}{c-m-1}\bigr)\neq 0$,
\item $h^n\bigl(\cE\tw{-(n+1)}{c-1}\bigr)\neq 0$,
\item $h^m(\cE\tw{-1}{-m})\neq 0$.
\end{enumerate}

We consider each of these possibilities in turn.

If (i) holds, then $\cO_X$ is a direct summand, as in the proof of
Theorem~\ref{thm:splittingO}.

If (ii) holds, then we consider \eqref{eq:Eulerexterior} tensored by
$\cE(-\Hyp-\fib)$, which gives
\begin{equation*}
0\to \cE\tw{-(n+1)}{c-1}\to \bigoplus_{i=0}^n
\cE\tw{-n}{c-a_i-1}\to
\cdots\to\bigoplus_{i=0}^n
\cE\tw{-1}{a_i-1} \to \cE(-\fib) \to 0.
\end{equation*}
Since by (d)
\[
H^n\bigl(\bigoplus_{i=0}^n
\cE\tw{-n}{c-a_i-1}\bigr)=\dots=H^1\bigl(\bigoplus_{i=0}^n
\cE\tw{-1}{a_i-1}\bigr) =0,
\]
we have a surjection $H^0(\cE(-\fib))\to
H^n\bigl(\cE\tw{-(n+1)}{c-1}\bigr)$. So $H^0(\cE(-\fib))\neq 0$, and
there exists a nonzero map $f\colon \cE \to \cO_X(\fib)$.

On the other hand $H^n\bigl(\cE\tw{-(n+1)}{c-1}\bigr)\imic
H^m(\cE^\vee(-m\fib))$. Then the exact sequence \eqref{eq:pbKoszul1}
tensored by $\cE^\vee$ reads
\[
0 \to \cE^\vee(-m\fib)\to \cE^\vee(-(m-1)\fib)^{e_m}
\to \cdots \to  (\cE^\vee)^{e_1} \to \cE^\vee(\fib)\to 0.
\]
But by (c)
\[
H^1(\cE^\vee)=\dots =H^m(\cE^\vee(-(m-1)\fib))=0,
\]
so we have a surjective map $H^0(\cE^\vee(\fib))\to
H^n\bigl(\cE\tw{-(n+1)}{c-1}\bigl)$. Therefore
$H^0(\cE^\vee(\fib))\neq 0$ and there exists a nonzero map $g\colon
\cO_X(\fib)\to \cE$.

Let us consider the following commutative diagram:
\[
\begin{CD}
H^n\bigl(\cE\tw{-(n+1)}{c-1}\bigr)\otimes H^m(\cE^\vee(-m\fib))
@>\sigma>> H^{n+m}\bigl(\cO_X\tw{-(n+1)}{c-1-m}\bigr)\\
@VVV @VVV\\
H^0(\cE(-\fib))\otimes H^1(\cE^{\vee}(-\fib))
@>>> H^1(\cO_X(-\fib)\otimes\cO_X(-\fib))\\
@VVV @VVV\\
H^0(\cE(-\fib))\otimes H^0(\cE^\vee(\fib))
@>\tau>> H^0(\cO_X(-\fib)\otimes\cO_X(\fib))\\
@| @|\\
\Hom(\cE,\cO_X(\fib))\otimes\Hom(\cO_X(\fib),\cE)
@. \Hom(\cO_X(\fib),\cO_X(\fib)).
\end{CD}
\]
The spaces in the right column are all $1$-dimensional. The map
$\sigma$ comes from Serre duality and it is not zero, the right
vertical map are isomorphisms and the left vertical maps are
surjective, so the map $\tau$ is also not zero. This means that the
the map $f\circ g\colon \cO_X(\fib) \to \cO_X(\fib)$ is non-zero and
hence it is an isomorphism. This isomorphism shows that $\cO_X(\fib)$
is a direct summand of $\cE$.

If (iii) holds then the exact sequence \eqref{eq:pbKoszul1} tensored
by $\cE(-\Hyp)$ reads
\[
0 \to \cE(-\Hyp-m\fib)\to \cE(-\Hyp-(m-1)\fib)^{e_m}\to \cdots \to
\cE(-\Hyp)^{e_1} \to \cE(-\Hyp+\fib)\to 0.
\]
Also, by (b),
\[
h^1(\cE(-\Hyp))=\dots =h^m(\cE(-\Hyp-(m-1)\fib))=0,
\]
so $h^0(\cE(-\Hyp+\fib))\neq 0$. By Serre duality,
$h^m(\cE(-\Hyp-m\fib))=h^n(\cE^\vee\tw{-n}{c-1})$. From the exact
sequence
\[ 
0\to \cE^\vee\tw{-n}{c-1}\to \bigoplus_{i=0}^n
\cE^\vee\tw{-(n-1)}{c-a_i-1}\to\cdots \to\bigoplus_{i=0}^n
\cE^\vee((a_i-1)\fib) \to \cE^\vee(\Hyp-\fib) \to 0
\]
we get also $h^0(\cE^\vee(\Hyp-\fib))\neq 0$, since by (d),
\begin{equation*}
  \begin{aligned}
h^n\bigl(\bigoplus_{i=0}^n
\cE^\vee\tw{-(n-1)}{c-a_i-1}\bigr)
&=h^1\bigl(\bigoplus_{i=0}^n\cE\tw{-2}{a_i-m}\bigr)=
\cdots=h^1\bigl(\bigoplus_{i=0}^n\cE^\vee((a_i-1)\fib)\bigr)\\
&=h^n\bigl(\bigoplus_{i=0}^n
\cE\tw{-(n+1)}{c-a_i-m}\bigr)=0.
  \end{aligned}
\end{equation*}
This shows, by the same argument as before, that $\cO_X(\Hyp-\fib)$ is
a direct summand of $\cE$.
\end{proof}

If $c=n+1$ then $X=\PP^n\times\PP^m$ and
Theorem~\ref{thm:splittingOfh} is an improvement of
\cite[Theorem~1.4]{bm2}. To see this, first note that in this case (c)
is implied by (a), because for $0\leq j\leq m-1$ we have
$h^{j+1}(\cE^\vee\tw{t}{j})=h^{m+n-j-1}(\cE\tw{-t-n}{-t+j-m})$. Moreover,
because $a_0=\dots =a_n=1$, if $|I|=i$ we get
\[
h^i(\cE\tw{t}{a_I-1})=h^i(\cE\tw{t}{t-1+i})=h^i(\cE\tw{t+1-i}{t})
\]
and
\[
h^i(\cE^\vee\tw{t}{a_I-1})=h^i(\cE^\vee\tw{t+1-i}{t})
=h^{m+n-i}(\cE\tw{-t-2-n+i}{-t-m-1}).
\]
\smallskip

In what follows we will use the following truncations of
\eqref{eq:Eulerexterior} for $i=1,\dots ,n$:

\begin{equation}\label{eq:trunc1}
\begin{aligned}
0&\to \cO_X\tw{-(n+1)}{c}\to
\bigoplus_{|I|=n}\cO_X\tw{-n}{a_I}\to\cdots\\
&\to\bigoplus_{|I|=i+1}
\cO_X\tw{-(i+1)}{a_I}\to \Omega_{X|\PP^m}^i\to 0
\end{aligned}
\end{equation}
and
\begin{equation}\label{eq:trunc2}
0\to \Omega_{X|\PP^m}^i \to \bigoplus_{|I|=i}
\cO_X\tw{-i}{a_I}\to\cdots \to\bigoplus_{|I|=1} \cO_X\tw{-1}{a_I}\to
\cO_X \to 0.
\end{equation}
Dualising \eqref{eq:trunc2} and tensoring with $\cO_X\tw{-(n+1)}{c}$
gives
\begin{equation*}
\begin{aligned}
0&\to \cO_X\tw{-(n+1)}{c}\to \bigoplus_{|I|=1}
\cO_X\tw{-n}{c-a_I}\to\cdots\\
&\to\bigoplus_{|I|=i}
\cO_X\tw{-(n+1-i)}{c-a_I}\to
(\Omega_{X|\PP^m}^i)^\vee\tw{-(n+1)}{c}\to 0,
\end{aligned}
\end{equation*}
so, using the obvious fact that $c-a_I=a_{I'}$ where
$I'=\{0,\ldots,n\}\smallsetminus I$, we get
\[
(\Omega_{X|\PP^m}^{i})^\vee\tw{-(n+1)}{c}\cong\Omega_{X|\PP^m}^{n-i}.
\]

\begin{theorem}\label{thm:indecomposibleregular}
Let $\cE$ be an indecomposible vector bundle on $X$ such that
$\Reg(\cE)=0$.  Suppose that the following vanishing occurs.
\begin{enumerate}[(a)]
\item $h^{n+j}\bigl(\cE\tw{-(n+1)}{c-j-1}\bigr)= 0$ for $1\leq j< m$,
\item $h^{i+j}\bigl(\cE\tw{-(i+1)}{i-j}\bigr)=
  h^{i+j}\bigl(\cE\tw{-(i+1)}{i-j+1}\bigr)=0$ for $1\leq j\leq m$ and
  $0\leq i< n$,
\item $h^{j+1}(\cE^\vee(-j\fib))=h^{j+1}(\cE(-\Hyp-j\fib))= 0$ for
  $0\leq j<m$,
\item $h^{|I|}\bigl(\cE\tw{-|I|}{a_I-1}\bigr)=
  h^{|I|}\bigl(\cE^\vee\tw{-|I|+1}{a_I-1}\bigr)=0$ if $1\leq |I|\leq
  n$,
\item
  $\begin{aligned}[t]
  h^k\bigl(\bigoplus\limits_{|I|=1-k+i}\cE\tw{-k}{k+1-a_I}\bigr)&=0
  \text{ for }0< k\le i <n \text{
    and}\\ h^k\bigl(\bigoplus\limits_{|I|=k+1}
  \cE^\vee\tw{-(k-1)}{a_I-i-1}\bigr)&=0\text{ for }0< k\le n-i <n.
\end{aligned}$ 
\end{enumerate}
Then $\cE\cong \cO_X$, or $\cE\cong\cO_X(\fib)$, or
$\cE\cong\cO_X(\Hyp-\fib)$, or $\cE\cong
\Omega_{X|\PP^m}^i\tw{i+1}{-(i+1)}$ with $1<i<n$.
\end{theorem}

\begin{proof}
Comparing the definition of regularity with (a) and the first part of
(b) we see that $\cE(-\Hyp)$ is not regular if and only if one of the
following conditions is satisfied:
\begin{enumerate}[(i)]
\item $h^{n+m}\bigl(\cE\tw{-(n+1)}{c-m-1}\bigr)\neq 0$,
\item $h^n\bigl(\cE\tw{-(n+1)}{c-1}\bigr)\neq 0$,
\item $h^m(\cE\tw{-1}{-m})\neq 0$, 
\item $h^{i+m}\bigl(\cE\tw{-(i+1)}{i-m}\bigr)\neq 0$ for some $i$ with
  $0<i<n$.
\end{enumerate}
Conditions (i)--(iii) were considered in
Theorem~\ref{thm:splittingOfh} (notice that the twists we used in
conditions (c), (d) in the proof of Theorem~\ref{thm:splittingOfh} are
exactly those of conditions (c), (d) and (e) with $i=k=1$).

So we consider case (iv). Suppose that $0<i<n$ and
$h^{i+m}\bigl(\cE\tw{-(i+1)}{i-m}\bigr)\neq 0$. Let us consider the
following exact sequence, obtained from the dual of \eqref{eq:trunc2}
tensored by $\cE\tw{-(i+1)}{a+1}$ and \eqref{eq:pbKoszul1} tensored by
$\cE\tw{-(i+1)}{i}$:
\begin{equation*}
  \begin{aligned}
0 &\to \cE\tw{-(i+1)}{i-m} \to \cE\tw{-(i+1)}{i+1-m}^{e_m}\to\dots \to
\cE\tw{-(i+1)}{i}^{e_1}\to\\
&\to\bigoplus_{|I|=1}\cE\tw{-i}{i+1-a_I}\to\dots\to
\bigoplus_{|I|=i}\cE\tw{-1}{i+1-a_I}\to
         [\Omega_{X|\PP^m}^i]^\vee\otimes\cE\tw{-(i+1)}{i+1} \to 0.
  \end{aligned}
\end{equation*}
The second part of (b) gives
\[
h^{i+m}\bigl(\cE\tw{-(i+1)}{i+1-m}\bigr)=\dots =
h^{i+1}\bigl(\cE\tw{-(i+1)}{i}\bigr)=0
\]
and (e) gives
\[
h^i\bigl(\bigoplus_{|I|=1}\cE\tw{-i}{i+1-a_I}\bigr)=\dots
=h^1\bigl(\bigoplus_{|I|=i}\cE\tw{-1}{i+1-a_I}\bigr),
\]
so we get $h^0\bigl(\cE\otimes
\left(\Omega_{X|\PP^m}^i\tw{i+1}{-(i+1)}\right)^\vee\bigr)\neq 0$.

By Serre duality, $h^{i+m}\bigl(\cE\tw{-(i+1)}{i-m}\bigr)=
h^{n-i}\bigl(\cE^\vee\tw{-(n-i)}{c-i-1}\bigr)$.  Let us consider the
exact sequence obtained from \eqref{eq:trunc1} tensored by
$\cE^\vee\tw{1+i}{-(i+1)}$:
\begin{equation*}
\begin{aligned}
0&\to \cE^\vee\tw{-(n-i)}{c-i-1}\to \bigoplus_{|I|=n}
\cE^\vee\tw{-(n-i-1)}{a_I-i-1}\to\dots\\
&\to\bigoplus_{|I|=i+1}
\cE^\vee((a_I-i-1)\fib))\to \cE^\vee\otimes
\Omega_{X|\PP^m}^i\tw{i+1}{-(i+1)})\to 0.
\end{aligned}
\end{equation*}
By (e),
\[
h^{n-i}\bigl(\bigoplus_{|I|=n}
\cE^\vee\tw{-(n-i-1)}{a_I-i-1}\bigr)=\dots
=h^1\bigl(\bigoplus_{|I|=i+1} \cE^\vee((a_I-i+1)\fib))\bigr)=0,
\]
so we also get
$h^0\bigl(\cE^\vee\otimes\Omega_{X|\PP^m}^{i}\tw{-(i+1)}{i+1}\bigr)\neq
0$. From this we conclude that $\cE\cong
\Omega_{X|\PP^m}^{i}\tw{-(i+1)}{i+1}$.
\end{proof}

If $c=n+1$, so $X=\PP^n\times\PP^m$, then
Theorem~\ref{thm:indecomposibleregular} is similar to \cite[Theorem
 3.5]{bm2}.

\begin{remark}\label{rem:brown_sayrafi}
A splitting criterion for vector bundles on scrolls is given
in~\cite[Theorem~1.5]{BS}, proved by a spectral sequence
argument. That result and ours do not seem to be directly
comparable. Theorem~\ref{thm:splittingO} and
Theorem~\ref{thm:splittingOfh} use less information about the
cohomology: in particular, only some vanishing, whereas the result
in~\cite{BS} assumes a priori agreement between the dimensions of all
the cohomology spaces of $\cE$ and a split bundle. On the other hand,
our results apply to a more restricted class of bundles.
\end{remark}

In the case of rational normal scrolls
$X=\PP(\cO_{\PP^1}(a_0)\oplus\dots \oplus\cO_{\PP^1}(a_n)) $, namely
when $m=1$, the above splitting criteria
(Theorems~\ref{thm:splittingO}--\ref{thm:indecomposibleregular})
become much simpler
(Corollaries~\ref{coro:rnssplittingO}--\ref{coro:rnsindecomposibleregular}).

\begin{corollary}\label{coro:rnssplittingO}
Let $\cE$ be a rank $r$ vector bundle on a rational normal scroll
$X$. Then the following conditions are equivalent:
\begin{enumerate}[(i)]
\item for any integer $t$ we have the vanishing
  \begin{enumerate}[(a)]
    \item $h^n(\cE\tw{t}{c-1})= 0$,
    \item $h^{i+j}(\cE\tw{t}{i-j})=0$ for $j=0,\,1$ and $0\leq i<n$
      but $(i,j)\neq(0,0)$;
  \end{enumerate}
\item there are $r$ integers $t_1, \dots, t_r$ such that $\cE\cong
  \bigoplus_{i=1}^r \cO_X(t_i\Hyp)$.
\end{enumerate}
\end{corollary}

\begin{corollary}\label{coro:rnssplittingOfh}
Let $\cE$ be a vector bundle on a rational normal scroll $X$.  Then
$E$ is a direct sum of line bundles $\cO_X$, $\cO_X(\fib)$ and
$\cO_X(\Hyp-\fib)$ with some twist $t\Hyp$ if and only if, for any
integer $t$,
\begin{enumerate}
\item[(b)] $h^{i+j}(\cE\tw{t}{i-j})= 0$ for $0\leq i< n$ and $j=0,\,1$
  but $(i,j)\not=(0,0),\,(0,1)$,
\item[(d)] $h^{|I|}(\cE\tw{t}{a_I-1})=h^{|I|}(\cE^\vee\tw{t}{a_I-1})=
  0$ if $1\leq |I|\leq n$.
\end{enumerate}
\end{corollary}

\begin{corollary}\label{coro:rnsindecomposibleregular}
Let $\cE$ be an indecomposible vector bundle on a rational normal
scroll $X$ with $\Reg(\cE)=0$.  Suppose that the following vanishing
occurs.
\begin{enumerate}
\item[(b)]
  $h^{i+1}\bigl(\cE\tw{-(i+1)}{i-1}\bigr)=h^{i+1}\bigl(\cE\tw{-(i+1)}{i}\bigr)=0$
  for $0\leq i< n$,
\item[(d)] $h^{|I|}\bigl(\cE\tw{-|I|}{a_I-1}\bigr)=
  h^{|I|}\bigl(\cE^\vee\tw{-|I|+1}{a_I-1}\bigr)=0$ if $1\leq |I|\leq
  n$,
\item[(e)] $\begin{aligned}[t]
  h^k\bigl(\bigoplus\limits_{|I|=1-k+i}\cE\tw{-k}{k+1-a_I}\bigr)&=0
  \text{ for }0< k\le i <n\text{ and
  }\\ h^k\bigl(\bigoplus\limits_{|I|=k+1}
  \cE^\vee\tw{-(k-1)}{a_I-i-1}\bigr)&=0\text{ for }0< k\le n-i <n.
\end{aligned}$
\end{enumerate}
Then $\cE\cong \cO_X$, or $\cE\cong\cO_X(\fib)$, or
$\cE\cong\cO_X(\Hyp-\fib)$, or $\cE\cong
\Omega_{X|\PP^m}^i\tw{i+1}{-(i+1)}$ with $1<i<n$.
\end{corollary}

\bibliographystyle{amsplain}

\end{document}